\documentclass[12pt,oneside,reqno,bibliography=totoc]{scrartcl}

\usepackage{kucuk_cihan}
\usepackage[affil-it]{authblk}

\usepackage{mathtools}

\author{C\.{i}han Bahran}
\affil{School of Mathematics, University of Minnesota\\
Minneapolis, MN 55455, USA}
\date{}
\usepackage{etoolbox}
 \usepackage{relsize}
 
\usepackage{bookmark}
\usepackage{textcomp}

\usepackage[titletoc]{appendix}
 
\usepackage{xstring} %for if then cases in new commands 

\AtBeginDocument{%
  \addtolength\abovedisplayskip{-0.5\baselineskip}%
  \addtolength\belowdisplayskip{-0.5\baselineskip}%
}
 
\title{The commuting complex of the symmetric group with bounded number of $p$-cycles}

\DeclareMathOperator{\FI}{\mathbf{FI}}

\DeclareMathOperator{\FB}{\mathbf{FB}}
\DeclareMathOperator{\co}{H}
\DeclareMathOperator{\chain}{C}

\DeclareMathOperator{\sh}{\FI_{\pmb{\sharp}}}

\newcommand{\weak}{\delta}

\newcommand{\sym}[1]{
\ifstrempty{#1}{\mathfrak{S}_{\bullet}}{\mathfrak{S}_{#1}}
}

\newcommand{\local}{h^{\text{max}}}

\newcommand{\clp}[1]{\mathbf{K}_p(#1)}
\newcommand{\clt}[1]{\mathbf{K}_2(#1)}

\newcommand{\abp}[1]{\mathcal{A}_p(#1)}

\newcommand{\lap}[1]{\Lambda_p(#1)}
\newcommand{\lat}[1]{\Lambda_2(#1)}

\newcommand{\bul}{\bullet}

\AtBeginDocument{%
   \def\MR#1{}
}

\makeatletter
\def\blfootnote{\gdef\@thefnmark{}\@footnotetext}
\makeatother

\begin{document}
\maketitle
\blfootnote{\textup{2010} \textit{Mathematics Subject Classification}.
Primary 20D30; Secondary 20B30, 20E15.} 
\blfootnote{\textit{Key words and phrases}. Quillen complex, symmetric group, highly acyclic, $\FI$-modules.}
\vspace{-0.8 in}
\begin{onecolabstract} 
 For a fixed prime $p$, we consider a filtration of the commuting complex of elements of order $p$ in the symmetric group $\sym{n}$. The filtration is obtained by imposing successively relaxed bounds on the number of disjoint $p$-cycles in the cycle decomposition of the elements. We show that each term in the filtration becomes highly acyclic as $n$ increases. We use $\FI$-modules in the proof.
\end{onecolabstract}
\vspace{0.4 in}

%{\tableofcontents}

%\section{Introduction}

The commuting graph $\lap{G}$ of a finite group $G$ at the prime $p$ is the graph whose vertices are elements of order $p$ in $G$ with edges connecting elements that commute. Writing $\mathbf{K}(\Lambda)$ for the clique complex of a graph $\Lambda$, the \textbf{commuting complex} of  $G$ is $\clp{G}:=\mathbf{K}(\lap{G})$. In other words, a $(k+1)$-simplex in  $\clp{G}$ is a subset of $G$ with size $k$, whose elements are of order $p$ that pairwise commute. 
%Given a nonempty finite simplicial complex $X$ which is not contractible, we define $\con(X)$ to be the smallest $k \geq -1$ such that $X$ is $k$-connected but not $(k+1)$-connected. By convention $\con(X) = \infty$ if $X$ is contractible, every nonempty simplicial complex is $(-1)$-connected, and $\con(\empt) = \eksi$. A sequence $\{X_{n}\}$ of simplicial complexes is called \textbf{highly connected} if  
%\begin{align*}
%  \lim_{n \rarr \infty} \con(X_{n}) = \infty \, .
%\end{align*}
%pp

 As a poset, $\clp{G}$ has an evident Galois connection with the Quillen poset $\abp{G}$ \cite[Section 2]{quillen-Ap} of nontrivial elementary abelian $p$-subgroups of $G$, yielding a homotopy equivalence. When $G$ is a finite group of Lie type, $\clp{G}$ is homotopy equivalent to the Tits building of $G$ \cite[Section 3]{quillen-Ap} (also see \cite[Remark 2.3(iv)]{tw-equiv}). For an arbitrary finite group $G$, the complex $\clp{G}$ (or different incarnations of its $G$-homotopy type \cite{tw-equiv}) contains significant information about the representations \cite{knorr-rob}, \cite{thev-k-euler}, \cite{balmer-spg}, \cite{grodal-endo} and cohomology \cite{webb-local}, \cite{webb-split}, \cite{grodal-higher}, \cite{flores-webb-split}, \cite{symonds-bredon} of $G$ and its $p$-local subgroups in characteristic $p$. For a ``big picture'' point of view as to how $\clp{G}$ fits in the finite group theory landscape, I recommend Webb's \cite{webb-sub} and Alperin's \cite{alperin-lie-approach} surveys. Smith's book \cite{smith-subgroup-cx} is a more recent and  extensive reference.

The focus of this paper is the case $G = \sym{n}$, the symmetric group on $n$ letters. For each $a \geq 0$, let us write $\lap{\sym{n},a}$ for the induced subgraph of the commuting graph $\lap{\sym{n}}$ on elements that can be written as a product of \textbf{at most $\pmb{a}$} number of disjoint $p$-cycles. For example $\{(1\,2)(3\,4), (1\,3)(2\,4), (1\,4)(2\,3), (5\,6)(7\,8)(9\, 10)\}$ is a clique in $\lat{\sym{10},3}$. The clique complexes $\clp{\sym{n},a} := \mathbf{K}(\lap{\sym{n},a})$ provide a natural filtration
\begin{align*}
 \empt = \clp{\sym{n},0} \subseteq \clp{\sym{n},1} \subseteq \cdots \subseteq \clp{\sym{n},\floor{n/p}} = \clp{\sym{n}} \, .
\end{align*}
For each $a \geq 1$, the complex $\clp{\sym{n},a}$ has dimension $\floor{n/p}-1$, with the $a=1$ case being the easiest in terms of the combinatorics involved. Still, $\clp{\sym{n},1}$ is already interesting. After Bouc's computation of the fundamental group $\pi_{1}(\clt{\sym{7},1}) = \zz / 3$ \cite[Proposition 3]{bouc-certain}, combinatorialists have found a wealth of torsion in similarly defined \emph{matching complexes}  \cite{shareshian-wachs-matching}, \cite{jonsson-exact}, \cite{jonsson-five-torsion}. Also, $\clp{\sym{n},1}$ has been used to great effect in understanding the more mysterious $\clp{\sym{n}}$ \cite{ksontini-simp}, \cite{ksontini-fund}, \cite{shareshian-Ap}, \cite{shareshian-wachs-Ap}.

We now fix some terminology. All of our homology groups are over $\zz$. We say that a simplicial complex $X$ is \textbf{$\pmb{k}$-acyclic} if the reduced homology $\wt{\co_{t}}(X_{n})$ vanishes for $t \leq k$. We call a sequence $\{X_{n}\}$ of simplicial complexes is \textbf{highly acyclic} if for each $k \geq 0$, the complex $X_{n}$ is $k$-acyclic when $n$ is sufficiently large. 

Ksontini showed that $\clp{\sym{n}}$ is connected for $n \geq 2p + 1$ \cite[Proposition 2.4]{ksontini-simp}, and $\clp{\sym{n}}$ is simply connected for $n \geq p^{2} + p + 1$ \cite[Theorems 5.2, 5.3]{ksontini-simp}, providing evidence for the following: 
\begin{conj} \label{sym-conj}
 The sequence $\{\clp{\sym{n}}\}$ is highly acyclic.
\end{conj}
This paper came out of an unsuccessful attempt at proving Conjecture \ref{sym-conj} and instead settling with the $\clp{\sym{n},a}$.
\begin{thmx} \label{main-thm}
 For each $a \geq 1$, the sequence $\{\clp{\sym{n},a}\}$ is highly acyclic. More precisely, $\clp{\sym{n},a}$ is $k$-acyclic for $n \geq 2(k+2)ap -1$.
\end{thmx}
Some low degree cases of Theorem \ref{main-thm} are known with sharp bounds. Ksontini showed that $\clp{\sym{n},1}$ is connected when $n \geq 2p+1$, see the first part of the proof of \cite[Proposition 2.4]{ksontini-simp}. For sharpness, note that $\clt{\sym{4},1}$ has three connected components. Again Ksontini showed that $\clp{\sym{n},1}$ is simply connected for $n \geq 3p + 2$ \cite[Proposition 4.1]{ksontini-simp}, where the $\clt{\sym{n},1}$ case is due to earlier work of Bouc \cite[Proposition 2]{bouc-certain}. That the $3p+2$ bound is sharp follows from another paper of Ksontini \cite[Lemma 3.5, Proposition 4.1]{ksontini-fund}. This seems to suggest that the vanishing slope in Theorem \ref{main-thm} could be halved.

Our point of view in studying the homology $\co_{k}(\clp{\sym{n}, a})$ will be to fix 
\begin{itemize}
 \item the prime $p$,
 \item the homological degree $k$,
 \item the maximum allowable number of $p$-cycles $a$ in a single element,
\end{itemize}
and then to vary $n$. The assignment $n \mapsto \clp{\sym{n},a}$ defines a functor $\clp{\sym{\bul},a}$ from $\FI$ (finite sets and injections) to simplicial complexes, or shortly an $\FI$\textbf{-complex}. Thus $\co_{k}(\clp{\sym{\bul},a})$ is an $\FI$-module. There is similarly an $\FI$-module $\co_{k}(\clp{\sym{\bul}})$, where there is no imposed bound on the number of $p$-cycles. For motivations and an introduction to $\FI$-modules, see the first two sections of Church--Ellenberg--Farb \cite{cef}. All of our $\FI$-modules will be over $\zz$.

For a given $\FI$-module $V$ and an injection $f\colon S \emb T$ between finite sets, we often write $f_{*}: V_{S} \rarr V_{T}$ for the transition map. We call an $\FI$-module $V$ \textbf{torsion} if for every finite set $S$ and $v \in V_{S}$, there exists an injection $f \colon S \emb T$ such that $f_{*}(a) = 0$. We use the term \textbf{finitely generated} for an $\FI$-module as defined in \cite[Definition 1.2]{cef}. The following is a basic, but quite an important observation:

\begin{prop} \label{fg-tors-ok}
Suppose that $\{X_{n}\}$ is a sequence of simplicial complexes such that the assignment 
\begin{align*}
 n \mapsto X_{n}
\end{align*}
extends to a $\FI$-complex $X_{\bul}$. If the $\FI$-module $\co_{k}(X_{\bul})$ is finitely generated and torsion for every $k \geq 0$, then the sequence $\{X_{n}\}$ is highly acyclic.
\end{prop}
\begin{proof}
 We want to show that for every $k$, the assignment $n \mapsto \co_{k}(X_{n})$ vanishes for all but finitely many values of $n$, which is the same thing with showing that the $\FI$-module $\co_{k}(X_{\bul})$ vanishes on sufficiently large finite sets. Fix $k$, and let $A$ be a finite list of generators for $\co_{k}(X_{\bul})$ such that each $a \in A$ lies in $\co_{k}(X_{S_{a}})$ for some finite set $S_{a}$. Because $\co_{k}(X_{\bul})$ is torsion, there exists injections $\iota^{a} \colon S_{a} \emb T_{a}$ such that $\iota^{a}_{*}(a) = 0$. Defining $N := \max\{|T_{a}|: a \in A\}$, we claim that for any finite set $S$ with at least $N$ elements we have $\co_{k}(X_{S}) = 0$. This is because $\co_{k}(X_{S})$ is generated as an abelian group by  
\begin{align*}
  \{f_{*}(a): a\in A, \, f \colon S_{a} \emb S\}
\end{align*}
and every such injection $f \colon S_{a} \emb S$ factors through $\iota^{a}$, hence $f_{*}(a) = 0$. 
\end{proof}

%For each $a \geq 1$, the assignment $n \mapsto \clp{\sym{n},a}$ defines an $\FI$-complex $\clp{\sym{},a}$. Varying $a$, we have an increasing sequence 
%\begin{align*}
% \empt = \clp{\sym{},0} \subseteq \clp{\sym{},1} \subseteq \clp{\sym{},2} \subseteq \cdots 
%\end{align*}
%which gives a filtration of $\clp{\sym{}}$.
%%whose union is $\clp{\sym{}}$. The filtration spectral sequence is of the form 
%%\begin{align*}
%%\mathlarger{ E^{1}_{x,y} = \co_{x+y}\left( \clp{\sym{},x}, \clp{\sym{},x-1} \right) \imp \co_{x+y} (\clp{\sym{}})}
%%\end{align*}

\begin{thm} \label{torsion}
 The $\FI$-modules $\co_{k}(\clp{\sym{\bul}})$ and $\co_{k}(\clp{\sym{\bul},a})$ are torsion for every $a,k \geq 1$.
\end{thm}
\begin{proof}
We shall prove the stronger claim that for any injection $f\colon S \emb T$ with $|T| - |S| \geq p$, the simplicial map $f_{*} \colon \clp{\sym{S},a} \rarr \clp{\sym{T},a}$ is null-homotopic. In particular, there is no element of these $\FI$-modules that survives transition maps more than $p$ degrees beyond.
 To see this, pick $B \subseteq T - f(S)$ with size $|B| = p$. Now pick a $p$-cycle $\sigma \in \sym{T}$ which permutes $B$ and leaves $T-B$ fixed. Because $f(S)$ and $B$ are disjoint, we have a well-defined order preserving map 
\begin{align*}
 \clp{\sym{S},a} &\rarr \clp{\sym{T},a} \\
 Q &\mapsto f_{*}(Q) \sqcup \{\sigma\} \, .
\end{align*}
The relations $f_{*}(Q) \subseteq f_{*}(Q) \sqcup \{\sigma\} \supseteq \{\sigma\}$ prove that $f_{*}$ is homotopic to the constant map $Q \mapsto \{\sigma\}$. The same argument works for $\clp{\sym{\bul}}$.
\end{proof}

\begin{prop} \label{chain-fg}
 For every $k \geq 0$ and $a \geq 1$, the $\FI$-module $\chain_{k}(\clp{\sym{\bul},a})$, defined by the $k$-th chain groups,
is generated in degrees $\leq (k+1)ap$.
\end{prop}
\begin{proof}
 One needs to pay attention to the degree shift in the clique complex construction: $\chain_{k}(\mathbf{K}(\Lambda))$ has the $(k+1)$-cliques of the graph $\Lambda$ as a basis. Thus $\chain_{k}(\clp{\sym{S},a})$ is the direct sum of subsets $Q \subseteq \sym{S}$ of size $k+1$, such that
\begin{itemize}
 \item every $\sigma \in Q$ is a product of at most $a$ disjoint $p$-cycles, and
 \item every $\sigma,\tau \in Q$ commute with each other.
\end{itemize}
Now if we were to write out all the elements in $Q$ in their cycle decompositions, the total number of symbols we see would be at most $(k+1)ap$. Thus if we write $B$ for the set of elements in $S$ that is moved by one of $\sigma \in Q$, then $|B| \leq (k+1)ap$ and there exists $Q_{B} \subseteq \sym{B}$ with the same size and properties above.  Therefore $Q_{B} \in \chain_{k}(\clp{\sym{A},a})$ and writing $\iota \colon B \emb S$ for the inclusion, we have $\iota_{*}(Q_{B}) = Q$. 
\end{proof}

At this point, the Noetherian property of $\FI$-modules due to Church--Ellenberg--Farb--Nagpal \cite[Theorem A]{cefn}, suffices to prove the high acyclicity part of Theorem \ref{main-thm}. We can also reduce Conjecture \ref{sym-conj} to an $\FI$-module statement:

\begin{conjbis}{sym-conj}
 There exists a chain complex of finitely generated $\FI$-modules $\mathbf{C_{*}}$ such that $\co_{k}(\mathbf{C_{*}}) = \co_{k}(\clp{\sym{\bul}})$.
\end{conjbis}

To get the explicit vanishing ranges in Theorem \ref{main-thm}, we will need to put in a little more work. Let us write $\sh$ for the category of \textbf{partial bijections}, as in \cite[Definition 4.1.1]{cef}. We call a finite filtration of an $\FI$-module a \textbf{finite $\sh$-filtration} if each factor $\FI$-module in the filtration extends to an $\sh$-module.

\begin{prop} \label{chain-has-filt}
 For every $k \geq 0 $ and $a \geq 1$, the $\FI$-module $\chain_{k}(\clp{\sym{\bul},a})$ has a finite $\sh$-filtration.
\end{prop}
\begin{proof}
We first show that the $\FI$-module 
\begin{align*}
 V := \bigoplus_{k=0}^{\infty} \chain_{k}(\clp{\sym{\bul},a})
\end{align*}
extends to an $\sh$-module. To that end, take a partial bijection 
$\phi \colon S \supseteq A \xleftrightarrow{f} B \subseteq T$ and $Q \subseteq \sym{S}$ of commuting elements which are products of at most $a$ disjoint $p$-cycles, noting that $V$ is spanned by such $Q$. Now we can simply declare $\phi_{*}(Q) := 0$ if $Q \cap \sym{A}$ is empty, and $\phi_{*}(Q) := f_{*}(Q \cap \sym{A})$, which is a similar set of elements in $\sym{T}$ with possibly smaller size than $Q$. Checking functoriality is straightforward.

Using the classification of $\sh$-modules obtained by Church--Ellenberg--Farb \cite[Theorem 4.1.5]{cef}, we have $V = M(W)$ for some $\FB$-module $W$: here $\FB$ is the category of finite sets and bijections, and $M$ is the left adjoint of the restriction functor $\lMod{\FB} \rarr \lMod{\FI}$. Conversely every $\FI$-module of the form $M(X)$ extends to an $\sh$-module. 
 
  A result of Ramos \cite[Proposition 2.18]{ramos-fig} says that $\co_{i}^{\FI}(V) = 0$ for every $i \geq 1$, where $\co_{0}^{\FI} \colon \lMod{\FI} \rarr \lMod{\zz}$ is a certain right exact functor (we do not need its definition) and $\{\co_{i}^{\FI} : i \geq 1 \}$ are its right derived functors. Being a direct summand of $V$, the finitely generated (Proposition \ref{chain-fg}) $\FI$-module $\chain_{k}(\clp{\sym{\bul},a})$ also vanishes under $\co_{i}^{\FI}$ for $i \geq 1$. Hence it has a finite $\sh$-filtration by the homological characterization of Ramos \cite[Theorem B]{ramos-fig}.
\end{proof}

\begin{proof}[Proof of \textbf{\emph{Theorem \ref{main-thm}}}]
 By the definition of homology, we have
\begin{align*}
\co_{k}(\clp{\sym{\bul},a}) &= \coker \left( \chain_{k+1}(\clp{\sym{\bul},a}) \longrightarrow V^{k}  \right) \, , \\
\text{where } V^{k} &:= \ker\big( \chain_{k}(\clp{\sym{\bul},a}) \rarr \chain_{k-1}(\clp{\sym{\bul},a}\big) \, 
\end{align*}
as $\FI$-modules. We also observe that we would get $V^{k+1}$ if we replaced  $\coker$ with $\ker$ above. All the chain modules involved have finite $\sh$-filtrations by Corollary \ref{chain-has-filt}, thus using the notation of Church--Miller--Nagpal--Reinhold \cite{cmnr-range}, their \emph{local degree} $\local = \max\{h^{i}: i \geq 0\}$ is equal to $-1$ \cite[Corollary 2.13]{cmnr-range} (originially due to Li--Ramos \cite[Theorem F]{li-ramos}), and their \emph{stable degree} $\weak$ is at most $(k+2)ap$, $(k+1)ap$, $kap$, respectively, by Proposition \ref{chain-fg} and \cite[Proposition 2.6(4)]{cmnr-range}. Thus  by the proof of \cite[Proposition 3.3]{cmnr-range} we get 
\begin{align*}
 h^{0}(\co_{k}(\clp{\sym{\bul},a})) 
 \leq \max \left\{h^{0}(V^{k}), -1,  h^{2}(V^{k+1}) \right\} \leq 2(k+2)ap - 2  
\end{align*}
Because $\co_{k}(\clp{\sym{\bul},a})$ is torsion by Corollary \ref{torsion}, we are done due to the definition of $h^{0}$ \cite[2.5]{cmnr-range}.
\end{proof}

\begin{rem}[\textbf{Kneser graphs and hypergraph matching complexes}]
 Given a finite set $S$, consider the graph $\text{Kneser}_{p}(S)$ whose elements are subsets of $S$ with size $p$ (here $p$ need not be a prime) such that the edges connect disjoint subsets. Note that $\text{Kneser}_{2}(S) = \clt{\sym{S},1}$. The clique complex $\text{M}_{p}(S) := \mathbf{K}(\text{Kneser}_{p}(S))$ is often referred to as a \textbf{hypergraph matching complex} in the combinatorics literature. Virtually the same arguments we used for proving Theorem \ref{main-thm} show that the complex $\text{M}_{p}(S)$ is $k$-acyclic if $|S| \geq 2(k+2)p - 1$. However, better vanishing ranges have been known for some time. The most recent of these (to my knowledge) is due to Athanasiadis, who showed \cite[Theorem 1.2]{athanasiadis} that $\text{M}_{p}(S)$ is $k$-acyclic if $|S| \geq (k+2)p + k + 1$. The graphs $\text{Kneser}_{p}(S)$ themselves have received recent attention \cite{ramos-white-FI-graph}, \cite{ramos-speyer-white-FI-set} with an $\FI$ point of view.
\end{rem}

\bibliographystyle{hamsalpha}
\bibliography{stable-boy}

\end{document}